\documentclass[11pt,a4paper,reqno,english]{amsart}
\title[]{On the supercritical Schr\"{o}dinger equation\\
on the exterior of a ball}
\usepackage{%
    amssymb%
    ,babel%
    ,csquotes%
    ,graphicx%
    ,mathrsfs%
    ,eucal%
    ,xcolor%
    ,enumitem%
    ,geometry%
    ,esint}
\usepackage[%
    pdfauthor={Piero D'Ancona}%
    ,colorlinks=true%
    ,linkcolor=magenta%
    ,citecolor=cyan%
    ,urlcolor=blue%
    ,hyperfootnotes=false%
]{hyperref}

\newcommand{\bra}[1]{\langle #1 \rangle}
\newcommand{\one}[1]{\mathbf{1}_{#1}}

\usepackage{stmaryrd}
\SetSymbolFont{stmry}{bold}{U}{stmry}{m}{n} % no font warning
\newenvironment{sproof}%
  {\baselineskip0pt\normalfont\footnotesize%
  \abovedisplayskip=1pt\abovedisplayshortskip=1pt%
  \belowdisplayskip=1pt\belowdisplayshortskip=1pt%
  \vskip2pt\par\noindent$\llbracket$\ignorespaces}%
  {\ignorespaces$\rrbracket$\vskip2pt}
\def\⟦{\begin{sproof}\ignorespaces}\def\⟧{\unskip\end{sproof}}

\numberwithin{equation}{section}
\newtheorem{theorem}{Theorem}[section]
\newtheorem{corollary}[theorem]{Corollary}

\newtheorem{proposition}[theorem]{Proposition}
\theoremstyle{remark}
\newtheorem{remark}[theorem]{Remark}

\theoremstyle{definition}
\newtheorem{definition}[theorem]{Definition}

\date{\today}

\author[P.~D'Ancona]{Piero D'Ancona}
\address{Piero D'Ancona: 
Dipartimento di Matematica\\
Sapienza Universit\`{a} di Roma\\
Piazzale A.~Moro 2\\
00185 Roma\\
Italy}
\email{dancona@mat.uniroma1.it}
\thanks{%
}
\subjclass[2020]{%
}
\keywords{%
}

\begin{document}

\begin{abstract}
  We consider the mixed problem on the exterior of the unit
  ball in $\mathbb{R}^{n}$, $n\ge2$,
  for a defocusing Schr\"{o}dinger equation with a power
  nonlinearity $|u|^{p-1}u$, with zero boundary data.
  Assuming that the initial data are non radial,
  sufficiently small
  perturbations of \emph{large} radial initial data,
  we prove that
  for all powers $p>n+6$ the solution
  exists for all times, its Sobolev norms do not inflate,
  and the solution is unique in the energy class.
\end{abstract}

\maketitle

%%% >>> INDEX (toc)
%\tableofcontents

% e_f_pre  <<<<<<<<< PREAMBOLO

\section{Introduction}\label{sec:intr}

The literature on the
defocusing semilinear Schr\"{o}dinger equation
\begin{equation}\label{eq:NLS}
  iu_{t}+\Delta u=|u|^{p-1}u,
  \qquad
  u(0,x)=u_{0}(x)
\end{equation}
is extensive and we
mention \cite{Cazenave03-a}, \cite{Tao06-a} and
\cite{Dodson19-a} for an
introduction and detailed bibliographies.
Restricting to large $H^{1}$ (energy class) data,
Problem \eqref{eq:NLS} is well posed for
large data below the critical value
$p<p_{0}(n)$ where
$p_{0}(n)=\frac{n+2}{n-2}$ for $n>2$ and $p_{0}(n)=+\infty$
if $n=1,2$ (\cite{GinibreVelo79-b}).
The problem is well posed also in the critical case
$p=p_{0}(n)$ as proved in a series of important papers
(\cite{Bourgain98-a},
\cite{Bourgain99-a},
\cite{Grillakis00-a},
\cite{Tao05-a},
\cite{CollianderKeelStaffilani08-a},
\cite{RyckmanVisan07-a},
\cite{Visan07-a}).
Well or ill posedness
in the \emph{supercritical} case $p>p_{0}(n)$
has been for many years a completely open problem.
A recent breaktrhough was obtained in
\cite{MerleRaphaelRodnianski19-a}, where finite time blow up
was established for a class of large, radially symmetric, 
localized initial data and suitable ranges of $(n,p)$.

Here we consider the supercritical case $p>p_{0}(n)$
from a different perspective. It is not difficult
to check that for radial data the first blow up must
occur at the origin, or, equivalently, that if the
solution remains bounded near the origin then no
blow up can occur. This is an immediate consequence
of the bound, valid for spherically symmetric functions,
\begin{equation}\label{eq:strausslemma}
  |x|^{\frac{n}{2}-1}|u(x)|\lesssim\|\nabla u\|_{L^{2}},
\end{equation}
usually called \emph{Strauss' Lemma}.
Inequality \eqref{eq:strausslemma}
is a special case of the family of inequalities
\begin{equation}\label{eq:nonradialstr}
  \textstyle
  |x|^{\frac np-\sigma}|u(x)|
  \lesssim
  \||D|^{\sigma}u\|_{L^{p}_{|x|}L^{r}_{\omega}},
  \qquad
  \frac{n-1}{r}+\frac1p<\sigma<\frac np
\end{equation}
(see \cite{DAnconaLuca12-a}), where
the norm $L^{p}_{|x|}L^{r}_{\omega}$ is an $L^{p}$ norm
in the radial direction of the $L^{r}_{\omega}$
norm on spheres centered at 0.

Exploiting the previous remark, one can remove the singularity,
by considering the mixed problem
\begin{equation}\label{eq:defocintro}
  iu_{t}+\Delta u=|u|^{p-1}u,
  \qquad
  u(0,x)=u_{0}(x),
  \qquad
  u(t,\cdot)\vert_{\partial \Omega}=0
\end{equation}
on the exterior of the unit ball
\begin{equation*}
  \Omega=\{x\in\mathbb{R}^{n}:|x|>1\}.
\end{equation*}
One obtains that for radial initial data the solution
must exist for all times and all values of $p,n$.
The precise statement is the following:

\begin{proposition}[]\label{pro:globalrad}
  Let $\Omega=\mathbb{R}^{n}\setminus \overline{B(0,1)}$, 
  $n\ge2$, $p>1$ and let
  $u_{0}\in H^{1}_{0}\cap H^{2}(\Omega)$
  be radially symmetric.
  Then the mixed problem \eqref{eq:defocintro}
  has a global solution 
  $u\in C^{2}L^{2}\cap C^{1}H^{1}_{0}\cap CH^{2}$,
  satisfying the conservation of mass
  $\|u(t)\|_{L^{2}}=\|u_{0}\|_{L^{2}}$ and of energy
  \begin{equation}\label{eq:conservs}
    \textstyle
    E(u(t)):=
    \frac12
    \int_{\Omega}|\nabla_{x} u(t)|^{2}dx
    +
    \frac{1}{p+1}
    \int_{\Omega}|u(t)|^{p+1}dx=E(u_{0})
  \end{equation}
  and the uniform bound
  \begin{equation}\label{eq:unibound}
    \|u\|_{L^{\infty}(\mathbb{R}^{+}\times \Omega)}
    \lesssim
    \|u_{0}\|_{H^{1}}.
  \end{equation}
  If $v\in C^{2}L^{2}\cap C^{1}H^{1}_{0}\cap CH^{2}$ is
  a second solution of \eqref{eq:defocintro}
  with the same data, which is radially symmetric 
  or, more generally, bounded on any strip 
  $0\le t\le T$, then $v \equiv u$.

  Assume in addition $p>2N$ for some integer $N\ge1$
  and $(u_{0},f)$ with $f(z)=|z|^{p-1}z$
  satisfy the nonlinear compatibility conditions of order $N$.
  Then
  $u\in C^{N}L^{2}(\Omega)$ and
  $u\in C^{k}(H^{2(N-k)}(\Omega)\cap H^{1}_{0}(\Omega))$ 
  for $0\le k\le N-1$.
\end{proposition}

The proof of Proposition \ref{pro:globalrad} is given in
Section \ref{sub:exisrad};
see Definition \ref{def:NLcc} for the meaning of
the nonlinear compatibility conditions.
In principle, one can not expect that the
radial solution thus constructed is unique; there might
exist other non radial solutions with the same data,
since the problem is supercritical.
Uniqueness is discussed in detail below.

Once a global radial solution is available, a natural
question concerns the stability for non radial perturbations
of the initial data. In view of the blow up
result mentioned above, one may expect that for large
non radial data the solution blows up also for the
supercritical exterior problem. However, 
using a pseudoconformal transform argument,
one verifies that radial solutions decay as $t\to+\infty$,
and the decay is good enough to work out a perturbative argument:

\begin{proposition}[Decay of the radial solution]
  \label{pro:decayrad}
  Let $u$ be the solution constructed in 
  Proposition \ref{pro:globalrad} for
  a radially symmetric
  $u_{0}\in H^{1}_{0}\cap H^{2}(\Omega)$.
  Assume in addition that $xu_{0}\in L^{2}(\Omega)$.
  Then, for all $t>0$, $|x|\ge1$,
  $u$ satisfies the decay estimate
  \begin{equation}\label{eq:decayrad}
    |u(t,x)|\le
    C(\|u_{0}\|_{H^{1}}+\|xu_{0}\|_{L^{2}})
    \cdot
    \bra{t}^{-1}|x|^{1-\frac n2}.
  \end{equation}
\end{proposition}

Proposition \ref{pro:decayrad} is proved in
Section \ref{sub:decay}.
If $u_{0}$ is smoother (and $xu_{0}\in L^{2}$)
then regularity propagates, and Sobolev norms do not
inflate but remain bounded for all times
(see Corollary \ref{cor:regul} in
Section \ref{sub:non_inflation}).

Since the radial solution decays, a perturbative argument
is sufficient to obtain the following global existence
result, proved in Section \ref{sec:main}, which
is the main result of the paper.
Consider again the problem on
$\mathbb{R}^{+}\times \Omega$
\begin{equation}\label{eq:pertpb}
  iv_{t}+\Delta v=f(v),
  \qquad
  v(0,x)=v_{0},
  \qquad
  v(t,\cdot)\vert_{\partial \Omega}=0.
\end{equation}
Then we have:

\begin{theorem}[]\label{the:main}
  Let $n\ge 2$, $p>n+6$, 
  and $u_{0}\in H^{2m}(\Omega)$
  with $m=\lfloor \frac n2\rfloor+1$ a radial
  function such that $(u_{0},f)$ with $f(z)=|z|^{p-1}z$
  satisfy the compatibility conditions of order $m$.
  Assume in addition that $xu_{0}\in L^{2}(\Omega)$.
  Then there exists 
  $\epsilon=\epsilon(u_{0})>0$ such that the following holds. 

  If $v_{0}\in H^{2m}(\Omega)$, with $(v_{0},f)$ satisfying
  the nonlinear compatibility conditions of order $N$,
  and $\|u_{0}-v_{0}\|_{H^{2m}}< \epsilon$,
  then Problem \eqref{eq:pertpb} has a global
  solution
  $v\in C^{m}L^{2}(\Omega)\cap
     C^{k}(H^{2(m-k)}(\Omega)\cap H^{1}_{0}(\Omega))$ 
  for $0\le k\le m-1$.
\end{theorem}

In the proof of Theorem \ref{the:main} we use
a Strichartz estimate for the exterior problem in
presence of an integrable potential. This is proved
in Section \ref{sec:stri_for_line} as a consequence
of the Strichartz estimates on the exterior of convex
obstacles proved in \cite{Ivanovici10a}.

It remains to consider the problem of uniqueness. 
A well known strategy allows to establish uniqueness
of energy class solutions to dispersive equations
under the assumption that a smooth solution exists
(weak--strong uniqueness).
This is particularly convenient in the present situation
since Theorem \ref{the:main} yields a smooth solution
to \eqref{eq:pertpb}. By adapting the arguments in 
\cite{Struwe06} we get:

\begin{theorem}[]\label{the:uniqueness}
  Suppose all the assumptions in Theorem \ref{the:main}
  are satisfied and let $v$ be the solution constructed
  there.
  Let $I$ be an open interval containing
  $[0,T]$ and
  $\underline{v}\in C(I;H^{1}_{0}(\Omega))\cap 
    C^{1}(I;L^{2}(\Omega))$
  a distributional solution for $0\le t\le T$
  to Problem \eqref{eq:pertpb} with the same initial data
  as $v$,
  which satisfies an energy inequality 
  $E(\underline{v}(t))\le E(\underline{v}(0))$
  (see \eqref{eq:conservs}). 
  Then we have $\underline{v}(t)=v(t)$ for $0\le t\le T$.
\end{theorem}

It is clear that a similar strategy can be applied to other 
dispersive equations. Indeed, for the supercritical nonlinear
wave equation on the exterior of a ball,
we proved the global well posedness for quasi radial initial 
data in the companion paper \cite{DAncona19-a}.
For wave equations the pseudoconformal transform is not available,
however the decay of radial solutions can be proved using
the Penrose transform. 

The plan of the paper is the following.
In Section \ref{sec:stri_for_line} we recall the linear
theory for exterior problems and we prove energy and
Strichartz estimates for derivatives of solutions, also
for Schr\"{o}dinger equations perturbed by a
well behaved potential $V(t,x)$.
Section \ref{sec:nonl_theo} is devoted to the general
nonlinear theory and local existence results.
In Section \ref{sec:radial} the global radial solution
is studied in detail. The main result,
Theorem \ref{the:main}, is proved in Section \ref{sec:main},
and the weak--strong uniqueness is proved in the
final Section \ref{sec:weak_stro_uniq}.

\section{The linear exterior problem}\label{sec:stri_for_line}

Consider the linear mixed problem
with $(t,x)\in\mathbb{R}\times \Omega$
\begin{equation}\label{eq:linear}
  i\partial_{t}u+\Delta u=F(t,x)
  \qquad
  u(0,x)=u_{0}(x),
  \qquad
  u(t,\cdot)\vert_{\partial \Omega}=0.
\end{equation}
Denote by $\Delta_{D}$ or simply $\Delta$
the selfadjoint Dirichlet Laplacian on $\Omega$ with domain
$H^{1}_{0}(\Omega)\cap H^{2}(\Omega)$,
by $\Lambda$ its 
(nonnegative selfadjoint) square root
\begin{equation*}
  \Lambda=(-\Delta_{D})^{1/2},
  \qquad
  D(\Lambda)=H^{1}_{0}(\Omega)
\end{equation*}
and by $e^{it \Delta}$ the flow defined via
the spectral theorem.
Note that $\Lambda^{2k}=(-\Delta)^{k}$ for integer $k\ge0$,
and 
\begin{equation*}
  D(\Lambda^{2k})=D(\Delta^{k})=
  \{f\in
  H^{2k}(\Omega):
  f,\Delta f,\dots ,\Delta^{k-1}f\in H^{1}_{0}(\Omega)
  \}, 
\end{equation*}
\begin{equation*}
  D(\Lambda^{2k+1})=
  \{f\in
  H^{2k+1}(\Omega):
  f,\Delta f,\dots ,\Delta^{k}f\in H^{1}_{0}(\Omega)
  \}
\end{equation*}
that is to say
\begin{equation*}
  \textstyle
  D(\Lambda^{k})=
  \{f\in H^{k}(\Omega) \colon \Delta^{j}f\in H^{1}_{0}(\Omega),
  \ 0\le 2j\le k-1\}.
\end{equation*}
Then for all data $u_{0}\in L^{2}(\Omega)$,
$F\in L^{1}_{loc}(\mathbb{R};L^{2}(\Omega))$ 
there exists a unique solution
$u(t,x)\in C(\mathbb{R};L^{2}(\Omega))$, which can
be written in the form
\begin{equation*}
  \textstyle
  u(t,x)=e^{it \Delta}u_{0}+
  i\int_{0}^{t}e^{i(t-s)\Delta}F(s,\cdot)dx.
\end{equation*}

To formulate estimates of $u$ we need some notations.
Given an interval $I \subseteq \mathbb{R}$, a Banach
space $X$ of functions on $\Omega$ and $T>0$, we shall write
\begin{equation*}
  L^{p}_{I}X=L^{p}(I;X),
  \qquad
  L^{p}_{T}X=L_{[0,T]}X,
  \qquad
  L^{p}X=L_{[0,+\infty)}X
\end{equation*}
with the obvious norms. Moreover we shall write
for $N\in \mathbb{N}$
\begin{equation*}
  \textstyle
  \|u\|_{X^{q,r;N}_{I}}=
  \sum_{j=0}^{N}\|\partial^{N-j}_{t}u\|_{L^{p}_{I}W^{2j,q}}
  =
  \sum_{2k+|\alpha|\le 2N}\|
    \partial_{t}^{k}\partial^{\alpha}_{x}u\|_{L^{p}_{I}L^{q}}
\end{equation*}
where $W^{k,q}=W^{k,q}(\Omega)$ is the usual Sobolev space 
with norm
$\sum_{|\alpha|\le k} \|\partial^{\alpha}u\|_{L^{q}}$
and $H^{k}=W^{{2,k}}$.
Note that the order of spatial regularity of functions
in $X^{q,r;N}_{I}$ is $2N$. When
$I=[0,T]$ or $I=[0,+\infty)$ we use the notations
\begin{equation*}
  X_{T}^{q,r;N}=X_{[0,T]}^{q,r;N}
  \qquad
  X^{q,r;N}=X_{[0,+\infty)}^{q,r;N}.
\end{equation*}

From the integral representation of $u$ and the unitarity of the group we have
\begin{equation}\label{eq:L2est}
  \|u\|_{L^{\infty}_{I}L^{2}}\le
  \|u_{0}\|_{L^{2}}+\|F\|_{L^{1}_{I}L^{2}}
\end{equation}
for any interval $I$ containing 0.
Higher regularity estimates require compatibility conditions.
Given the data $(u_{0},F)$ we define
recursively a sequence of functions $h_{j}$ as follows:
\begin{equation}\label{eq:lincc}
  h_{0}=u_{0},\qquad
  h_{j}(x)=\partial_{t}^{j}u(0,x)=i^{-1}
  (\partial^{j-1}_{t}F(0,x)-\Delta h_{j-1}(x))
  \qquad
  j\ge1.
\end{equation}
An explicit computation gives
$h_{j}=(-i \Delta)^{j}u_{0}-i\sum_{\ell=0}^{j-1}
  (-i \Delta)^{j-\ell-1}\partial^{\ell}_{t}F(0,x)$.
% Now, differentiating the equation w.r.to time we get
% \begin{equation*}
%   u_{0}\in H^{1}_{0}\cap H^{2}(\Omega),
%   \quad
%   F\in C^{1}L^{2}(\Omega)
%   \quad\implies\quad
%   u\in C H^{1}_{0}(\Omega)\cap H^{2}(\Omega)
% \end{equation*}
% and the estimate
% \begin{equation*}
%   \|u_{t}\|_{L^{\infty}_{I}L^{2}}+
%   \|\Delta u\|_{L^{\infty}_{I}L^{2}}\lesssim
%   \|h_{1}\|_{L^{2}}+
%   \|F_{t}\|_{L^{1}_{I}L^{2}}
% \end{equation*}
\begin{definition}[Linear compatibility conditions]\label{def:lcc}
  We say that the data $(u_{0},F)$ satisfy the
  \emph{linear compatibility conditions of order $N\ge1$} if
  $u_{0}\in H^{2N}(\Omega)\cap H^{1}_{0}(\Omega)$,
  $F\in C^{k}H^{2(N-k-1)}(\Omega)\cap C^{N}L^{2}(\Omega)$ 
  for $0\le k\le N-1$, and
  \begin{equation}\label{eq:Lcc}
    h_{j}\in H^{1}_{0}(\Omega)
    \quad\text{for}\quad 
    0\le j\le N-1.
  \end{equation}
\end{definition}

Then one has the following standard result.

\begin{theorem}[]\label{the:linearthm}
  Assume $(u_{0},F)$ satisfy the
  linear compatibility conditions of order $N$ for some $N\ge1$.
  Then the global solution $u$ to Problem \eqref{eq:linear}
  satisfies $u\in C^{N}L^{2}(\Omega)$ and
  $u\in C^{k}(H^{2(N-k)}(\Omega)\cap H^{1}_{0}(\Omega))$ 
  for $0\le k<N$.
  Moreover, for any interval $I \subseteq \mathbb{R}$
  containing 0 and of length $\gtrsim1$ we have the estimate
  \begin{equation}\label{eq:estfree}
    \|u\|_{X^{\infty,2;N}_{I}}\le C(N)
    \left[
      \|u_{0}\|_{H^{2N}}+\|F\|_{X^{1,2;N}_{I}}
    \right].
  \end{equation}
\end{theorem}

\begin{proof}%[Proof of ...]
  The result is classical (see e.g.~\cite{Tsutsumi83-b})
  and is proved by differentiating the equation w.r.to time
  and estimating spatial derivatives inductively.
  We depart from standard results only in the
  formulation of the energy estimate
  \eqref{eq:estfree}: the right
  hand side is usually expressed in the form
  \begin{equation*}
    \textstyle
    \|F\|_{X^{\infty,2;N-1}_{I}}+
    \|\partial_{t}^{N}F\|_{L^{1}_{I}L^{2}},
  \end{equation*}
  which can be estimated by the $X^{1,2;N}_{I}$ norm 
  of $F$ since
  \begin{equation*}
    \|G\|_{L^{\infty}_{I}L^{2}}\lesssim
    \|G\|_{L^{1}_{I}L^{2}}+
    \|\partial_{t}G\|_{L^{1}_{I}L^{2}}.
    \qedhere
  \end{equation*}
  % provided the interval $I$ has length $\gtrsim1$.
\end{proof}

Consider next the equation with a time dependent potential
$V(t,x)$
\begin{equation}\label{eq:linearV}
  i\partial_{t}u+\Delta u=V(t,x)u+F(t,x),
  \qquad
  u(t_{0},x)=f(x),
  \qquad
  u(t,\cdot)\vert_{\partial \Omega}=0.
\end{equation}
If we assume that
for some interval $I$ containing 0 (possibly $I=\mathbb{R}$)
\begin{equation*}
  u_{0}\in L^{2},
  \qquad
  F\in C_{I}L^{2},
  \qquad
  V\in L^{1}_{I}L^{\infty},
\end{equation*}
then the existence of a unique solution
$u\in C_{I}L^{2}$ is proved by a simple 
contraction argument
for the map $v \mapsto u$, where $u$ is defined as the
solution to
\begin{equation*}
  i\partial_{t}u+\Delta u=V(t,x)v+F(t,x),
  \qquad
  u(t_{0},x)=f(x),
  \qquad
  u(t,\cdot)\vert_{\partial \Omega}=0,
\end{equation*}
followed by a continuation argument.
The solution satisfies the estimate
\begin{equation}\label{eq:L2Vest}
  \|u\|_{L^{\infty}_{I}L^{2}}\le
  C(\|V\|_{L^{1}_{I}L^{\infty}})
  \left[
    \|u_{0}\|_{L^{2}}+\|F\|_{L^{1}_{I}L^{2}}
  \right].
\end{equation}
Note that it is not necessary to modify the compatibility
conditions in the higher regularity case. Indeed, the
conditions should be
\begin{equation*}
  \textstyle
  h_{j}=
  i^{-1}
  (\partial^{j-1}_{t}F(0,x)+
  \sum_{k=0}^{j-1}\binom {j-1}{k}
  \partial_{t}^{j-1-k}V(0,x)h_{k}
  -\Delta h_{j-1}(x)),
  \qquad
  j\ge1
\end{equation*}
but if the potential $V$ is sufficiently smooth, the
term $\partial_{t}^{j-1-k}V(0,x)h_{k}$ belongs to
$H^{1}_{0}$ by the recursive assumption and can be
omitted.

We denote the solution of \eqref{eq:linearV},
with initial data at $t=t_{0}$, by
\begin{equation*}
  u(t,x)=S(t;t_{0})f.
\end{equation*}
Note that $S(t;t_{0})$ fails to be a group since
the potential $V$ depends on time.
Regarding the
term $Vu$ as a forcing term and applying Duhamel's formula
we can write $S(t;t_{0})$ as a perturbation of the free flow:
\begin{equation}\label{eq:duham}
  \textstyle
  S(t;t_{0})=
  e^{i(t-t_{0})\Delta}-
  i\int_{t_{0}}^{t}e^{i(t-s)\Delta}V(s,x)S(s,t_{0})ds.
\end{equation}
We impose a rather restrictive condition on $V$ in order
to obtain a uniform energy estimate for all times:
\begin{equation}\label{eq:Vass}
  \textstyle
  \|V\|_{X^{1,\infty;N}_{I}}=
  \sum_{j=0}^{N}\|\partial_{t}^{N-j}V\|
  _{L^{1}_{I}W^{2j,\infty}}
  <\infty.
\end{equation}
Note that by standard embeddings we have
\begin{equation*}
  \textstyle
  \|V\|_{X^{1,\infty;N}_{I}}\lesssim
  \|V\|_{X^{1,2;N+N_{2}}_{I}},
  \qquad
  N_{2}=\lfloor \frac n2 \rfloor+1.
\end{equation*}
By the usual recursive argument 
(repeated differentiation with respect to time)
we obtain the regularity result:

\begin{theorem}[Perturbed energy estimate]\label{the:enerV}
  Assume $(u_{0},F)$ satisfy the compatibility
  conditions of order $N$ for some integer $N\ge1$.
  Then problem  \eqref{eq:linearV} has a unique global
  solution, which
  satisfies $u\in C^{N}L^{2}(\Omega)$ and
  $u\in C^{k}(H^{2(N-k)}(\Omega)\cap H^{1}_{0}(\Omega))$ 
  for $0\le k<N$.
  Moreover, for any interval $I \subseteq \mathbb{R}$
  containing 0 
  % and of length $\gtrsim1$ we have
  \begin{equation}\label{eq:estfreeV}
    \|u\|_{X^{\infty,2; N}_{I}}\le 
    C(N,\|V\|_{X^{1,\infty;N}_{I}})
    \left[
      \|u_{0}\|_{H^{2N}}+\|F\|_{X^{1,2;N}_{I}}
    \right].
  \end{equation}
\end{theorem}

We next consider the decay properties of the linear solution.
Strichartz estimates for the exterior problem
are available from \cite{Ivanovici10a}.
Recall that a couple of indices $(q,r)$
is \emph{admissible} if $(q,r)\in [2,\infty]\times[2,\infty)$ 
and it satisfies the scaling condition
\begin{equation*}
  \frac 2q+\frac nr=\frac n2.
\end{equation*}
The \emph{endpoint} is the couple $(2,\frac{2n}{n-2})$;
the restriction $r<\infty$
means that the endpoint is not admissible when $n=2$.
Note that in the following result
$\Omega$ could be more generally
the exterior of any smooth convex obstacle in $\mathbb{R}^{n}$.

\begin{theorem}\label{the:strichext}
  Let $n\ge2$, $(q,r)$ a non endpoint
  admissible couple, and $I \subseteq \mathbb{R}$ an interval
  containing 0. 
  % of length $\gtrsim1$.
  Assume $(u_{0},F)$ satisfy the
  linear compatibility conditions of order $m$ for some $m\ge1$.
  Then the solution $u$ to Problem \eqref{eq:linear} satisfies
  the Strichartz estimate
  \begin{equation}\label{eq:strich}
    \|u\|_{X^{q,r;m}_{I}}\lesssim\|u_{0}\|_{H^{2m}}+
    \|F\|_{X^{1,2;m}_{I}}.
  \end{equation}
\end{theorem}

\begin{proof}%[Proof of ...]
  We can assume $I=\mathbb{R}$.
  If $m=0$ and $F=0$,
  the result is Theorem 1.7 in \cite{Ivanovici10a},
  and if $F$ is nonzero the result follows by a
  standard Christ--Kiselev argument.
  If $m=1$, we apply $\partial_{t}$
  and use the estimate just obtained; this gives
  \begin{equation*}
    \|u_{t}\|_{L^{q}L^{r}}\lesssim
    \|\Delta u_{0}\|_{L^{2}}+\|F_{t}\|_{L^{1}L^{2}}
    \le 
    \|\Delta u_{0}\|_{L^{2}}+\|F\|_{X^{1,2;1}}.
  \end{equation*}
  Since $\Delta u=F-iu_{t}$ this implies
  \begin{equation}\label{eq:estDu}
    \|\Delta u\|_{L^{q}L^{r}}\lesssim
    \|\Delta u_{0}\|_{L^{2}}+\|F\|_{X^{1,2;1}}+
    \|F\|_{L^{q}L^{r}}.
  \end{equation}
  Note that $\|F\|_{L^{\infty}L^{2}}\lesssim\|F\|_{X^{1,2;1}}$;
  % (in the case of an interval $I$, this requires the 
  % assumption $|I|\gtrsim1$); 
  moreover, $X^{1,2;1}$ embeds into $L^{1}H^{2}$ and into
  $L^{\infty}L^{2}$, hence by complex interpolation it
  embeds into $L^{2}H^{1}$.
  If $n\ge3$, this embeds into the endpoint
  $L^{2}L^{\frac{2n}{n-2}}$ and hence in all admissible
  spaces $L^{q}L^{r}$ by interpolation with the
  embedding into $L^{\infty}L^{2}$.
  This proves the $L^{q}L^{r}$ estimate for $\Delta u$,
  and by $L^{q}$ elliptic regularity
  (see e.g.~\cite{Browder60-a}) this gives \eqref{eq:strich}
  for $m=1$, $n\ge1$. The same argument works for
  the case $m=1$, $n=2$ using the embedding 
  $H^{1}\hookrightarrow BMO$. Finally,
  for larger values of $m>1$, the usual recursion argument
  and the embedding 
  $X^{1,2;m}\hookrightarrow X^{q,r;m-1}$ just proved for 
  admissible $(q,r)$ allows to conclude the proof.
\end{proof}

Combining \eqref{eq:strich} with the perturbed energy
estimate \eqref{eq:estfree}, we obtain a similar result for the
perturbed linear problem
\begin{equation}\label{eq:linearVnh}
  i\partial_{t}u+\Delta u-V(t,x)u=F,
  \qquad
  u(t_{0},x)=f(x),
  \qquad
  u(t,\cdot)\vert_{\partial \Omega}=0.
\end{equation}

\begin{proposition}[Perturbed Strichartz estimate]
  \label{pro:strichV}
  Let $n\ge2$, $(q,r)$ a non endpoint
  admissible couple, and $I \subseteq \mathbb{R}$ an interval
  containing 0.
  Assume $(u_{0},F,V)$ satisfy the perturbed
  compatibility conditions of order $m$ for an $m\ge1$.
  Then the solution $u$ to Problem \eqref{eq:linearVnh} satisfies
  the Strichartz estimate
  \begin{equation}\label{eq:strichV}
    \|u\|_{X_{I}^{q,r;m}}\le
    C(m,p,\|V\|_{X_{I}^{1,\infty;m}})
    \cdot
    \left[
    \|u_{0}\|_{H^{2m}}+
    \|F\|_{X_{I}^{1,2;m}}
    \right].
  \end{equation}
\end{proposition}

\begin{proof}%[Proof of ...]
  Write
  $u=e^{it \Delta}u_{0}+ i\int_{0}^{t}e^{i(t-s)\Delta}(Vu+F)dx$
  and apply \eqref{eq:strich} to get
  \begin{align*} % &(all)&(separa)&(all)-piu' col meglio alignat
    \|u\|_{X_{I}^{q,r;m}}  &  
    \lesssim
    \|u_{0}\|_{H^{2m}}+\|Vu\|_{X_{I}^{1,2;m}}+
    \|F\|_{X_{I}^{1,2;m}}
  \\ %\intertext{thistextdoesnotinterruptalignment}
      &  \lesssim
    \|u_{0}\|_{H^{2m}}+\|V\|_{X_{I}^{1,\infty;m}}
    \|u\|_{X_{I}^{\infty,2;m}}+
    \|F\|_{X_{I}^{1,2;m}}.
  \end{align*}
  Using \eqref{eq:estfree} we obtain \eqref{eq:strichV}.
\end{proof}

\begin{remark}[]\label{rem:uubar}
  Note that the previous arguments 
  can be obviously applied without modification
  to the more general equations of the form
  \begin{equation*}
    iu_{t}+\Delta u=V_{1}(t,x)u+V_{2}(t,x)\overline{u}.
  \end{equation*}
\end{remark}

\section{The nonlinear theory}\label{sec:nonl_theo}

We consider now the nonlinear
mixed problem on $\mathbb{R}^{+}\times \Omega$
\begin{equation}\label{eq:NLpb}
  iu_{t}+\Delta u=f(u),
  \qquad
  u(0,x)=u_{0},
  \qquad
  u(t,\cdot)\vert_{\partial \Omega}=0.
\end{equation}
The compatibility conditions must be modified as follows.
Define recursively the sequence of functions 
$\psi_{j}(x)$ for $j\ge0$ as
\begin{equation*}
  \psi_{0}=u(0,x)=u_{0},
  \qquad
  \psi_{j}=\partial^{j}_{t}u(0,x)=
  \iota^{-1}(\partial_{t}^{j-1}f(u)\vert_{t=0}-
    \Delta \psi_{j-1}(x))
\end{equation*}
where in the expansion of $\partial_{t}^{j-1}f(u)$
we replace $\partial^{k}_{t}u(0,x)$ with $\psi_{k}$ for $k<j$.
% Explicitly,
% \begin{equation*}
%   \textstyle
%   \partial_{t}^{j-1}f(u)\vert_{t=0}=
%   \sum\limits_{\nu=1}^{j-1}
%   f^{(\nu)}(\psi_{0})\!\!
%   \sum\limits_{ \substack{ j_{i}>0\\\sum j_{i}=j-1}}
%   \frac{(j-1)!}{j_{1}!\dots j_{\nu}!}
%   \psi_{j_{1}}\dots \psi_{j_{\nu}}.
% \end{equation*}

\begin{definition}[Nonlinear compatibility conditions]
\label{def:NLcc}
  We say that the data $(u_{0},f)$ satisfy the
  \emph{nonlinear compatibility conditions of order}
  $N\ge1$ if $u_{0}\in H^{2N}(\Omega)\cap H^{1}_{0}(\Omega)$,
  $f\in C^{N}$, $f(0)=0$ and
  \begin{equation}\label{eq:NLcc}
    \psi_{j}\in H^{1}_{0}(\Omega)
    \quad\text{for}\quad 
    0\le j\le N-1.
  \end{equation}
\end{definition}

We get a local smooth solution to \eqref{eq:NLpb}
by a standard contraction argument:

\begin{proposition}[Local existence in $H^{\frac n2+}$]
\label{pro:local}
  Assume $(u_{0},f)$ in \eqref{eq:NLpb} satisfy the
  nonlinear compatibility condition of order $N$ for some
  integer $N>n/2$. Then there exists a time $T>0$
  and a unique solution of \eqref{eq:NLpb} on
  $[0,T]\times \Omega$ such that
  $u\in C^{N}([0,T];L^{2}(\Omega))$ and
  $u\in C^{k}([0,T];H^{2(N-k)}(\Omega)\cap H^{1}_{0}(\Omega))$ 
  for $0\le k\le N-1$.
\end{proposition}

\begin{proof}%[Proof of ...]
  Let $T>0$ and denote by $Z_{T}$ the space of functions
  $v(t,x)$ such that
  \begin{equation*}
    v\in\cap_{k=0}^{N-1}C^{k}([0,T];H^{2(N-k)}(\Omega)\cap
      H^{1}_{0}(\Omega))
      \cap C^{N}([0,T];L^{2}(\Omega)),
    \quad
    v(0,x)=u_{0}(x)
  \end{equation*}
  endowed with the metric
  $d(v,w)=\|v-w\|_{X_{T}^{\infty,2;N}}$.
  Consider the linearized problem
  \begin{equation}\label{eq:linearized}
    iu_{t}+\Delta u=f(v(t,x)),
    \qquad
    u(0,x)=u_{0},
    \qquad
    u(t,\cdot)\vert_{\partial \Omega}=0.
  \end{equation}
  If $(u_{0},f)$ satisfy \eqref{eq:NLcc}
  and $v(t,x)\in Z_{T}$, 
  then setting $F(t,x)=f(v(t,x))$ we see that
  the data $(u_{0},F)$ satisfy the linear compatibility condition
  of order $N$ \eqref{eq:Lcc} (note that $u_{0}$ is bounded
  since $2N>n/2$).
  Thus the solution of \eqref{eq:linearized} 
  is uniquely defined on $[0,T]\times \Omega$ and has the
  properties listed in Theorem \ref{the:linearthm}.
  Hence the map $\Phi: v \mapsto u$ which takes
  $v$ into the solution $u$ of the linearized problem 
  \eqref{eq:linearized} operates on the metric space $Z_{T}$.
  If $v_{1},v_{2}\in Z_{T}$ then $w=\Phi(v_{1})-\Phi(v_{2})$
  solves the problem
  \begin{equation*}
    iw_{t}+\Delta w=f(v_{1})-f(v_{2}),
    \qquad
    w(0,x)=0,
    \qquad
    w(t,\cdot)\vert_{\partial \Omega}=0.
  \end{equation*}
  Since $(0,G)$ with $G=f(v_{1})-f(v_{2})$ satisfy the
  linear compatibility conditions, we can apply 
  \eqref{eq:estfree}:
  \begin{equation*}
    \|\Phi(v_{1})-\Phi(v_{2})\|_{X_{T}^{\infty,1;N}}
    \le C(N)\|f(v_{1})-f(v_{2})\|_{X_{T}^{1,2;N}}
    \le C(N) T 
    \|f(v_{1})-f(v_{2})\|_{X_{T}^{\infty,2;N}}.
  \end{equation*}
  Moreover, for $N>n/2$ we have the Moser type estimates
  \begin{equation*}
    \|f(v)\|_{X_{T}^{\infty,2;N}}
    \le \phi_{N}(\|v\|_{X_{T}^{\infty,2;N}}),
  \end{equation*}
  \begin{equation*}
    \|f(v_{1})-f(v_{2})\|_{X_{T}^{\infty,2;N}}
    \le \phi_{N}(\|v_{1}\|_{X_{T}^{\infty,2;N}}+
    \|v_{2}\|_{X_{T}^{\infty,2;N}})
    \|v_{1}-v_{2}\|_{X_{T}^{\infty,2;N}}
  \end{equation*}
  where $\phi_{N}$ is a suitable nondecreasing function
  depending only on $N$ and $f$.

  Now let $\underline{v}(t,x)=u_{0}(x)$ for all $t$;
  note that $\underline{v}\in Z_{T}$, and let
  $B_{M}$ be the closed ball in $Z_{T}$ centered at $\underline{v}$
  of radius $M$.
  By the previous estimates, it is trivial to check
  that $\Phi:B_{M}\to B_{M}$ 
  and $\|\Phi(v_{1})-\Phi(v_{2})\|_{Z_{T}}\le C(M)T\le \frac 12$
  provided $M$ is large enough w.r.to
  $\|\underline{v}\|_{Z_{T}}=\|u_{0}\|_{H^{2N}}$ and 
  $T$ is sufficiently small w.r.to $M$.
  A contraction argument then implies the claim.
\end{proof}

\section{The radial solution}\label{sec:radial}

\subsection{Existence: Proof of Proposition \ref{pro:globalrad}}
\label{sub:exisrad}

Consider the equation
\begin{equation}\label{eq:defoc}
  iu_{t}+\Delta u=|u|^{p-1}u,
  \qquad
  u(0,x)=u_{0}(x),
  \qquad
  u(t,\cdot)\vert_{\partial \Omega}=0
\end{equation}
with \emph{radial} initial data 
$u_{0}\in H^{2}(\Omega)\cap H^{1}_{0}(\Omega)$.
The existence of a local solution for $t\in[0,T]$ is standard
and similar to the proof of Proposition \ref{pro:local}.
We use the space $Y_{T}$
of functions $v(t,x)$, radial in $x$, such that
\begin{equation*}
  v\in 
  C^{1}([0,T];L^{2}(\Omega))
      \cap C([0,T];H^{1}_{0}(\Omega) \cap H^{2}(\Omega)),
  \qquad
  v(0,x)=u_{0}.
\end{equation*}
We endow $Y_{T}$ with the distance 
$d(v,w)=\|v-w\|_{X_{T}^{\infty,2;1}}$.
By radiality and \eqref{eq:strausslemma}, functions in $Y_{T}$
are bounded. Hence the map $\Phi:v \mapsto u$,
defined as above via the linearization \eqref{eq:linearized},
operates on $Y_{T}$ and satisfies
\begin{equation*}
  \|u\|_{X_{T}^{\infty,2;1}}\lesssim
  \|u_{0}\|_{H^{2}}+\||v|^{p-1}v\|_{X_{T}^{1,2;1}}
  \lesssim \|u_{0}\|_{H^{2}}
  + \|v\|_{L^{\infty}_{T}L^{\infty}}^{p-1}
  \|v\|_{X_{T}^{1,2;1}}.
\end{equation*}
Estimate \eqref{eq:strausslemma} implies that
$\|v\|_{L^{\infty}_{T}L^{\infty}}
  \lesssim\|v\|_{X_{T}^{\infty,2;1}}$
so that
\begin{equation*}
  \|\Phi(v)\|_{X_{T}^{\infty,2;1}}\lesssim
  \|u_{0}\|_{H^{2}}
  + \|v\|_{X_{T}^{\infty,2;1}}^{p-1}
  \|v\|_{X_{T}^{1,2;1}}
  \lesssim
  \|u_{0}\|_{H^{2}}
  +T\|v\|_{X_{T}^{\infty,2;1}}^{p}.
\end{equation*}
In a similar way,
\begin{equation*}
  \|\Phi(v_{1})-\Phi(v_{2})\|_{X_{T}^{\infty,2;1}}\lesssim
  (
  \|v_{1}\|_{X_{T}^{\infty,2;1}}+
  \|v_{2}\|_{X_{T}^{\infty,2;1}})^{p-1}
  \cdot T\|v_{1}-v_{2}\|_{X_{T}^{\infty,2;1}}^{p}.
\end{equation*}
Thus taking $T$ sufficiently small with respect to
$\|u_{0}\|_{H^{2}}$ we obtain a local solution with
the required regularity. 

The local solution satisfies the energy conservation
\eqref{eq:conservs}, which implies the uniform bound
\eqref{eq:unibound} for $0\le t\le T$.
In particular, $\|u(t)\|_{H^{2}\cap L^{p+1}}$
remains bounded on $[0,T]$ by a constant depending
only on $\|u_{0}\|_{H^{2}\cap L^{p+1}}$.
A standard continuation argument allows to extend $u$
to a global solution, which satisfies \eqref{eq:conservs}
and hence \eqref{eq:unibound} for all times.
The claim about uniqueness follows immediately from
energy estimates.

Assume now that the data satisfy the compatibility conditions
of order $N\ge2$. Differentiating the equation w.r.to $t$
we see that $v=\partial_{t}u$ 
satisfies an equation of the form
\begin{equation*}
  iv_{t}+\Delta v=a(t,x)v+b(t,x)\overline{v}
\end{equation*}
with $|a|+|b|\lesssim|u|^{p-1}$ bounded.
By the linear theory we get $v\in C L^{2}$ 
i.e.~$u\in C^{1}L^{2}$.
Further differentiating w.r.to $t$, by a recursive argument
we get $u\in C^{N}L^{2}$, and using the equation itself we
obtain that 
$u\in C^{k}(H^{2(N-k)}(\Omega)\cap H^{1}_{0}(\Omega))$ 
for $0\le k\le N-1$.

\subsection{Decay: Proof of Proposition \ref{pro:decayrad}}
\label{sub:decay}

Apply to $u$ the pseudoconformal transform
\begin{equation*}
  \textstyle
  u(t,x)=t^{-\frac n2}U(-\frac 1t,\frac xt)
  e^{\frac{i|x|^{2}}{4t}}.
\end{equation*}
If $u(t,x)$ is defined on the domain $t\ge1$,
$|x|\ge1$, then $U(T,X)$ is defined for
$-1\le T<0$, $|X|\ge|T|$ and we have
\begin{equation*}
  \textstyle
  iu_{t}+\Delta_{x} u-|u|^{p-1}u=
  t^{-\frac n2-2}
  (iU_{T}+\Delta_{X}U-(-T)^{\nu}|U|^{p-1}U)
  e^{\frac{i|x|^{2}}{4t}}
\end{equation*}
with $\nu=\frac n2(p-1)-2$.
The energy density
\begin{equation*}
  e_{0}(T,X)=
  \textstyle
  \frac 12|\nabla_{X}U|^{2}+\frac{1}{p+1}|U|^{p+1}
\end{equation*}
satisfies the identity
\begin{equation*}
  \textstyle
  \partial_{T}e_{0}=
  \Re \nabla \cdot\{\overline{U}_{T}\nabla U\}
  -
  \frac{\nu(-T)^{\nu-1}}{p+1}|U|^{p+1}.
\end{equation*}
Since $U$ is a radial function we have
\begin{equation*}
  R^{n-1}\nabla_{X}\cdot\{\overline{U}_{T}\nabla U\}=
  \partial_{R}\{R^{n-1}U_{R}\overline{U}_{T}\}
\end{equation*}
where $\partial_{R}U=U_{R}=\frac{X}{|X|}\cdot \nabla_{X} U$
denotes the radial derivative of $U$ and $R=|X|$.
Introduce the radial energy density
\begin{equation*}
  \textstyle
  e(T,R)=R^{n-1}
  \left(
    \frac 12|U_{R}|^{2}+\frac{1}{p+1}|U|^{p+1}
  \right)
\end{equation*}
so that
\begin{equation}\label{eq:intermen}
  \partial_{T}e(T,R)=
  \Re \partial_{R}\{R^{n-1}U_{R}\overline{U}_{T}\}
  -
  \textstyle
  \frac{\nu(-T)^{\nu-1}}{p+1}R^{n-1}|U(T)|^{p+1}.
  % \le
  % \Re \partial_{R}\{R^{n-1}U_{R}\overline{U}_{T}\}.
\end{equation}
We now integrate the identity
\eqref{eq:intermen} on the $(T,R)$ domain
$T_{1}\le T\le T_{2}$, $R\ge -T$
for some $-1\le T_{1}<T_{2}<0$.
Note that the exterior normal on the line $R=-T$ 
(for $T<0$) is given by
$\mathbf{n}=(-\frac{1}{\sqrt{2}}, -\frac{1}{\sqrt{2}})$.
Writing
\begin{equation*}
  \mathcal{E}(T)=\int_{-T}^{+\infty}e(T,R)dR
\end{equation*}
after integration of \eqref{eq:intermen} we obtain
\begin{equation*}
\begin{split}
  \textstyle
  \mathcal{E}(T_{2})-\mathcal{E}(T_{1})
  -\frac{1}{\sqrt{2}}\int_{T_{1}}^{T_{2}}
  &
  e(T,-T)dT
  =
  \\
  =&
  \textstyle
  -
  \frac{1}{\sqrt{2}}
  \Re\int_{T_{1}}^{T_{2}}U_{R}\overline{U}_{T}|T|^{n-1}dT
  -
  \int_{T_{1}}^{T_{2}}\int_{-T}^{+\infty}
  \frac{\nu(-T)^{\nu-1}|U|^{p+1}}{p+1}dRdT
\end{split}
\end{equation*}
since $-T=|T|=R$ at the points of the cone.
Writing (with a slight abuse) $U(T,|X|)$ instead of
$U(T,X)$, the Dirichlet condition implies
$U(T,-T)=0$, so that
\begin{equation*}
  \textstyle
  e(T,-T)=|T|^{n-1}\frac 12|U_{R}|^{2}.
\end{equation*}
On the other hand, differentiating $U(T,-T)=0$ we get
$U_{T}(T,-T)=U_{R}(T,-T)$.
Hence the previous estimate reduces to
\begin{equation*}
  \textstyle
  \mathcal{E}(T_{2})-\mathcal{E}(T_{1})
  +
  \int_{T_{1}}^{T_{2}}\int_{-T}^{+\infty}
  \frac{\nu(-T)^{\nu-1}|U|^{p+1}}{p+1}dRdT
  \le
  \frac{1}{\sqrt{2}}
  \int_{T_{1}}^{T_{2}}
  \left(
    \frac{|U_{R}|^{2}}{2}-|U_{R}|^{2}
  \right)
  |T|^{n-1}dT
  \le
  0.
\end{equation*}
Thus $\mathcal{E}(T)$ is nonincreasing as $T \uparrow 0$ and in
particular $\mathcal{E}(T)\le\mathcal{E} (-1)$ 
for $-1<T<0$. This implies
\begin{equation*}
  \textstyle
  \int_{-T}^{+\infty}|U_{R}|^{2}R^{n-1}dR\le 2\mathcal{E}(-1)
  \quad\text{for}\quad -1<T<0
\end{equation*}
or equivalently
\begin{equation*}
  \textstyle
  \int_{|X|>-T}|\nabla_{X}U|^{2}dX\le 2\mathcal{E}(-1)
  \quad\text{for}\quad -1<T<0.
\end{equation*}
Note that, if we extend $U$ as zero in the region $|X|<-T$,
the extended function $\widetilde{U}(T,\cdot)$ is 
$H^{1}(\mathbb{R}^{n})$
and radial, hence we can apply \eqref{eq:strausslemma}
and we obtain
\begin{equation}\label{eq:strE}
  \textstyle
  |U(T,X)|^{2}\lesssim
  |X|^{2-n}
  \int|\nabla_{X}\widetilde{U}|^{2}dX
  =
  |X|^{2-n}
  \int_{|X|>-T}|\nabla_{X}U|^{2}dX
  \le
  2|X|^{2-n}\mathcal{E}(-1).
\end{equation}
We convert \eqref{eq:strE} into an estimate for $u(t,x)$. We have
\begin{equation*}
  \textstyle
  \mathcal{E}(-1)=\int_{|X|>1}
  \left(
    \frac{|\nabla U(-1,X)|^{2}}{2}+
    \frac{|U(-1,X)|^{p+1}}{p+1}
  \right)
  dX.
\end{equation*}
Since
\begin{equation*}
  \textstyle
  U(T,X)=(-T)^{-\frac n2}u(-\frac 1T,\frac XT)
  e^{\frac{i|X|^{2}}{4T}}
\end{equation*}
we compute
\begin{equation*}
  \textstyle
  \nabla_{X}U(T,X)=
  (-T)^{-\frac n2}
  e^{\frac{i|X|^{2}}{4T}}
  \left[
    \nabla_{x}u(-\frac 1T,\frac XT) \cdot\frac{1}{T}+
    u(-\frac 1T,\frac XT) \cdot i \frac{X}{2T}
  \right]
\end{equation*}
so that
\begin{equation*}
  \textstyle
  |\nabla_{X}U(-1,X)|\le
  |\nabla u(1,-X)|+|Xu(1,-X)|
\end{equation*}
and
\begin{equation}\label{eq:estE1}
  \textstyle
  \mathcal{E}(-1)\le
  \int_{|x|>1}
  \left[
    |\nabla u(1,x)|^{2}+|xu(1,x)|^{2}+
    |u(1,x)|^{p+1}
  \right]dx.
\end{equation}
On the other hand, changing variables
$(t,x)=(-\frac 1T,\frac XT)$ in $\mathcal{E}(T)$ 
and writing 
\begin{equation*}
  \textstyle
  \mathcal{E}_{1}(t)=\mathcal{E}(-\frac 1t)
\end{equation*}
we get, after a standard computation,
\begin{equation*}
  \textstyle
  \mathcal{E}_{1}(t)=
  \int_{|x|>1}
  \left[
    \frac 18|(x+2it \nabla)u(t,x)|^{2}+
    \frac{|t|^{\frac n2(p-1)}}{p+1}|u(t,x)|^{p+1}
  \right]
  dx
\end{equation*}
and $\mathcal{E}_{1}(t)$ is nonincreasing in $t$ by the
previous computation; note this is a proof of the
pseudoconformal energy conservation on an exterior domain.
Hence we have
\begin{equation*}
  \|xu(t)\|_{L^{2}(\Omega)}^{2}
  \lesssim
  \mathcal{E}_{1}(t)+
  (t+t^{\frac n2(p-1)})E(u(t))\le
  \mathcal{E}_{1}(0)+\bra{t}^{\frac n2(p-1)}E(u_{0})
\end{equation*}
and in conclusion
\begin{equation*}
  \|xu(t)\|_{L^{2}(\Omega)}^{2}
  \le
  C(t)
  \left[
    \|xu_{0}\|_{L^{2}(\Omega)}^{2}+
    E(u(0))
  \right].
\end{equation*}
Combined with the conservation of $E(u(t))$ and \eqref{eq:estE1}
this gives
\begin{equation*}
  \mathcal{E}(-1)
  \lesssim
  \|xu_{0}\|_{L^{2}(\Omega)}^{2}+ E(u(0))
\end{equation*}
and, recalling \eqref{eq:strE}, we have proved
\begin{equation*}
  |U(T,X)|\lesssim|X|^{1-\frac n2}
  \left[
    \|xu_{0}\|_{L^{2}(\Omega)}+ E(u(0))^{1/2}
  \right].
\end{equation*}
Writing $|U(T,X)|=|T|^{-n/2}|u(-\frac 1T,\frac XT)|$ we finally
obtain
\begin{equation*}%
  \textstyle
  t^{\frac n2}|u(t,x)|
  \le C|x|^{1-\frac n2}
  \cdot t^{\frac n2-1}
  \left[
    \|xu_{0}\|_{L^{2}(\Omega)}+ E(u(0))^{1/2}
  \right]
\end{equation*}
that is to say
\begin{equation}\label{eq:decayu}
  |u(t,x)|\le
  C
  \left[
    \|xu_{0}\|_{L^{2}(\Omega)}+ E(u(0))^{1/2}
  \right]
  \cdot|x|^{1-\frac n2}t^{-1}.
\end{equation}
Note that, using \eqref{eq:strausslemma}, we have
\begin{equation*}
  \textstyle
  E(u(0))=\frac 12\|\nabla_{x}u_{0}\|_{L^{2}}^{2}
  + \frac{1}{p+1}\int|u_{0}|^{p+1}\le
  \|\nabla_{x}u_{0}\|_{L^{2}}^{2}+
  \|u_{0}\|_{L^{\infty}}^{p-1}\int|u_{0}|^{2}
  \le C(\|u_{0}\|_{H_{1}}).
\end{equation*}
Using \eqref{eq:decayu} for $t>1$, and the inequality
\begin{equation*}
    |u(t,x)|\le C|x|^{1-\frac n2}\|\nabla_{x}u\|_{L^{2}}
\end{equation*}
for $t\le1$, we obtain as claimed
\begin{equation*}
  |u(t,x)|\le
  C
  \cdot|x|^{1-\frac n2}\bra{t}^{-1}.
\end{equation*}
with a constant depending on
$\|xu_{0}\|_{L^{2}(\Omega)}+ \|u_{0}\|_{H^{1}}$.

\subsection{Non inflation of Sobolev norms}
\label{sub:non_inflation}

We now consider the issue of regularity; 
we prove that if the data are smoother
the solution remains smooth and Sobolev norms remain bounded,
for any order of regularity.

\begin{corollary}[]\label{cor:regul}
  Let $N\ge1$, $p> 2N-1$, and assume $(u_{0},f)$
  with $f(z)=|z|^{p-1}z$ 
  satisfy the compatibility conditions of order $N$.
  Assume in addition that $xu_{0}\in L^{2}(\Omega)$.
  Then the radial solution constructed in 
  Proposition \ref{pro:globalrad} satisfies the 
  uniform bound on $\mathbb{R}^{+}\times \Omega$
  \begin{equation}\label{eq:boundsHN}
    \|u(t,\cdot)\|_{X^{\infty,2;N}}
    \le C(\|u_{0}\|_{H^{2N}},\|xu_{0}\|_{L^{2}}).
  \end{equation}
\end{corollary}

Recall that for all times $t$ we have
$\|u(t)\|_{L^{2}}=\|u_{0}\|_{L^{2}}$.
We next prove
uniform bounds for the derivatives of $u$;
in the following proof we write for brevity
\begin{equation*}
  C(\|u_{0}\|_{H^{M}})=
  C(\|u_{0}\|_{H^{M}},\|xu_{0}\|_{L^{2}}).
\end{equation*}
leaving the dependence on $\|xu_{0}\|_{L^{2}}$ implicit.

Differentiating the equation once with respect
to $t$ we see that $v=u_{t}$ solves
\begin{equation*}
  iv_{t}+\Delta v=\partial_{t}(|u|^{p-1}u)
\end{equation*}
and we have by \eqref{eq:decayrad}
\begin{equation*}
  \textstyle
  \|\partial_{t}(|u|^{p-1}u)\|_{L_{T}^{1}L^{2}}
  \lesssim
  \int_{0}^{T}
    \|u\|_{L^{\infty}}^{p-1}\|v\|_{L^{2}} dt
  \lesssim
  \int_{0}^{T}\bra{t}^{1-p}\|v\|_{L^{2}} dt.
\end{equation*}
Thus we can write
\begin{equation*}
  \textstyle
  \|v\|_{L^{\infty}_{T}L^{2}}=
  \|u_{t}\|_{L^{\infty}L^{2}}\lesssim
  \|u_{t}(0)\|_{L^{2}}+
  \int_{0}^{T}\bra{t}^{1-p}\|v\|_{L^{2}} dt
\end{equation*}
and by Gronwall's Lemma, if $p>2$ we get
\begin{equation*}
  \|v\|_{L^{\infty}L^{2}}=
  \|u_{t}\|_{L^{\infty}L^{2}}
  \le
  C
  \|u_{t}(0)\|_{L^{2}}
\end{equation*}
with $C=C(\|u_{0}\|_{H^{1}}+\|xu_{0}\|_{L^{2}})$;
since $|u_{t}(0)|\le|\Delta u_{0}|+|u_{0}|^{p}$ we have
\begin{equation*}
  \|u_{t}(0)\|_{L^{2}}\le
  \|u_{0}\|_{H^{2}}+\|u_{0}\|_{L^{\infty}}^{p-1}\|u_{0}\|_{L^{2}}
  \le C(\|u_{0}\|_{H^{2}})
\end{equation*}
which gives the estimate
\begin{equation*}
  \|u_{t}\|_{L^{\infty}L^{2}}\le C(\|u_{0}\|_{H^{2}}).
\end{equation*}
Using the equation for $u$ we can estimate also
\begin{equation*}
  \|\Delta u\|_{L^{\infty}L^{2}}\le
  \|u_{t}\|_{L^{\infty}L^{2}}+
  \||u|^{p}\|_{L^{\infty}L^{2}}
  \le  C(\|u_{0}\|_{H^{2}})+
  \|u\|_{L^{\infty}L^{\infty}}^{p-1}\|u_{0}\|_{L^{2}}
  \le C(\|u_{0}\|_{H^{2}}).
\end{equation*}
By elliptic regularity we have then
\begin{equation*}
  \|u\|_{L^{\infty}H^{2}}\le C(\|u_{0}\|_{H^{2}})
\end{equation*}
and summing up we have proved
\begin{equation}\label{eq:firstder}
  \|u\|_{X^{\infty,2;1}}\le C(\|u_{0}\|_{H^{2}}).
\end{equation}
For higher order derivatives we proceed in a similar
way by induction.
Applying $\partial_{t}^{j}$ to the equation we have
\begin{equation*}
  i(\partial^{j}_{t}u)_{t}+
  \Delta(\partial^{j}_{t}u)=
  \partial^{j}_{t}(|u|^{p-1}u).
\end{equation*}
The right hand side satisfies
\begin{equation*}
  |\partial^{j}_{t}(|u|^{p-1}u)|
  \lesssim
  \bra{t}^{1-p}|\partial^{j}_{t}u|+
  L.O.T.
\end{equation*}
where the lower order terms are products of derivatives
$\partial^{h}_{t}u$ with $h<j$, which are bounded by
the induction hypothesis (recall that $u$ is radial in $x$,
hence $L^{\infty}$ norms are bounded by $L^{2}$ norms of
the gradient),
times a power of $u$ of order at least $p-j$, which
decays like $\bra{t}^{j-p}$ and is integrable provided
$p>j+1$. This implies
\begin{equation*}
  \textstyle
  \|L.O.T.\|_{L^{1}_{T}L^{2}}\lesssim
  \int_{0}^{T}\bra{t}^{p-j}C(\|u_{0}\|_{H^{2j}})dt
  \lesssim C(\|u_{0}\|_{H^{2j}})
\end{equation*}
and hence
\begin{equation*}
  \textstyle
  \|\partial^{j}_{t}u\|_{L^{\infty}_{T}L^{2}}
  \le \|\partial^{j}_{t}u(0)\|_{L^{2}}+
  \int_{0}^{T}\|\partial^{j}_{t}(|u|^{p-1}u)\|_{L^{2}}dt
\end{equation*}
\begin{equation*}
  \textstyle
  \lesssim
  \|\partial^{j}_{t}u(0)\|_{L^{2}}+
  C(\|u_{0}\|_{H^{2j}})+
  \|\bra{t}^{1-p}\partial^{j}_{t}u\|_{L^{1}_{T}L^{2}}.
\end{equation*}
Thus again by Gronwall's Lemma, if $p>j+1$, we have
\begin{equation}\label{eq:derj}
  \|\partial^{j}_{t}u\|_{L^{\infty}L^{2}}
  \le C(\|u_{0}\|_{H^{2j}}).
\end{equation}
Using the equation, this shows that if $p>N+1$ one has
\begin{equation*}
  \|\partial^{N}_{t}u\|_{L^{\infty}L^{2}}+
  \|\partial^{N-1}_{t}\Delta u\|_{L^{\infty}L^{2}}
  \le  C(\|u_{0}\|_{H^{2j}})
\end{equation*}
and by elliptic regularity we have also
\begin{equation*}
  \textstyle
  \|\partial^{N}_{t}u\|_{L^{\infty}L^{2}}+
  \|\partial^{N-1}_{t} u\|_{L^{\infty}H^{2}}
  \le  C(\|u_{0}\|_{H^{2j}})
\end{equation*}
provided $p>N+1$.
Derivatives $\partial^{j}_{t}\partial^{\alpha}_{x}u$
can be estimated using the equation and elliptic regularity 
by the usual recursive procedure;
this requires to estimate
$\Delta^{j}(|u|^{p-1}u)$ with $j=1,\dots,N-1$, 
and in order to get an integrable factor $\bra{t}^{p-2j}$
we must assume $p-2(N-1)>1$ i.e. $p>2N-1$.
We finally arrive at the estimate
\begin{equation}\label{eq:noninfl}
  \|u\|_{X^{\infty,2;N}}\le C(\|u_{0}\|_{H^{2N}})
\end{equation}
provided $p>2N-1$.

\section{Proof of Theorem \ref{the:main}}\label{sec:main}

Denote by $u(t,x)$ the global radial solution with 
$u(0,x)=u_{0}$
constructed in
Propositions \ref{pro:globalrad}, \ref{pro:decayrad}
and Corollary \ref{cor:regul},
and let $v(t,x)$ be the local solution with $v(0,x)=v_{0}$
given by Proposition \ref{pro:local}.
Then $w=v-u$ satisfies the equation
\begin{equation*}
  iw_{t}+\Delta w=|u+v|^{p-1}(u+v)-|u|^{p-1}u
\end{equation*}
which can be written in the form
\begin{equation}\label{eq:diffeq}
  iw_{t}+\Delta w-V(t,x)w=w^{2}\cdot F[u,w],
  \qquad
  w(0)=w_{0}:=v_{0}-u_{0},
  \qquad
  w\vert_{\partial \Omega}=0
\end{equation}
where
\begin{equation*}
  V(t,x)=p|u|^{p-1},
  \qquad
  \textstyle
  F[u,w]=
  p(p-1)
  \int_{0}^{1}|u+\sigma w|^{p-3}(u+\sigma w)(1-\sigma)d \sigma.
\end{equation*}
We have
\begin{equation*}
  \textstyle
  \|V\|_{X_{T}^{1,\infty;m}}\lesssim
  \sum_{2j+|\alpha|\le 2m}
  \|\partial^{j}_{t}\partial^{\alpha}|u|^{p-1}\|
    _{L^{1}_{T}L^{\infty}}
\end{equation*}
and if $p-1>2m$ we can write
\begin{equation*}
  \|\partial^{j}_{t}\partial^{\alpha}|u|^{p-1}\|
    _{L^{1}_{T}L^{\infty}}
  \lesssim
  \textstyle
  \sum
  \int_{0}^{T}
  \|u\|_{L^{\infty}}^{p-1-\nu}
  \|\partial^{j_{1}}_{t}\partial^{\alpha_{1}}u\|_{L^{\infty}}
  \dots
  \|\partial^{j_{\nu}}_{t}\partial^{\alpha_{\nu}}u\|_{L^{\infty}}
  dt
\end{equation*}
\begin{equation*}
  \lesssim
  \textstyle
  \sum
  \|u\|_{X^{\infty,\infty;m}}^{p-1-\nu}
  \cdot
  \int_{0}^{T}
  \|u\|_{L^{\infty}}^{p-1-\nu}dt
\end{equation*}
where the sum is extended over $j_{1}+\dots +j_{\nu}=j$,
$\alpha_{1}+\dots +\alpha_{\nu}=\alpha$ and
$\nu\le j+|\alpha|$, so that $\nu\le 2m$.
By the decay estimate \eqref{eq:decayrad} 
we have $\|u\|_{L^{\infty}}\lesssim \bra{t}^{-1}$
hence the last integral is convergent provided
$p>2m+2$. By Sobolev embedding and the bound \eqref{eq:boundsHN}
we get
\begin{equation*}
  \textstyle
  \|u\|_{X^{\infty,\infty;m}_{T}}\lesssim
  \|u\|_{X^{\infty,2;m+N_{2}}_{T}}\lesssim
  C(\|u_{0}\|_{H^{2(m+N_{2})}},\|xu_{0}\|_{L^{2}}),
  \qquad
  N_{2}=\lfloor  \frac n2  \rfloor+1
\end{equation*}
and in conclusion we have proved
\begin{equation}\label{eq:estV}
  \|V\|_{X^{1,\infty;m}}
  \lesssim
  C(\|u_{0}\|_{H^{2(m+N_{2})}},\|xu_{0}\|_{L^{2}})
  <\infty
\end{equation}
provided $p>2m+2$. 
Thus we are in position to apply Theorem \ref{the:enerV} and
Proposition \ref{pro:strichV}, 
and we get that the linear equation
\begin{equation*}
  iw_{t}+\Delta w-V(t,x)w=F(t,x)
\end{equation*}
satisfies for all admissible non endpoint $(q,r)$
and all $m\ge1$ the perturbed energy--Strichartz estimates
\begin{equation}\label{eq:enstrV}
  \|w\|_{X^{\infty,2;m}_{T}}
  +
  \|w\|_{X^{q,r;m}_{T}}
  \lesssim \|w_{0}\|_{H^{2m}}+\|F\|_{X^{1,2;m}_{T}}
\end{equation}
provided $p>2m+2$, $u_{0}\in H^{2(m+N_{2})}$
and compatibility conditions of suitable order are satisfied.
The implicit constant in \eqref{eq:enstrV} depends on
$\|u_{0}\|_{H^{1}}+\|xu_{0}\|_{L^{2}}$ but not on $T$.

We now apply \eqref{eq:enstrV} to the equation
\eqref{eq:diffeq}. For an admissible couple $(q,r)$
and an integer $m$ to be chosen we write
\begin{equation*}
  M_{m}(T)=  \|w\|_{X^{\infty,2;m}_{T}}
  +
  \|w\|_{X^{q,r;m}_{T}}
\end{equation*}
and by \eqref{eq:enstrV} we have
\begin{equation}\label{eq:basic}
  M_{m}(T)\lesssim\|w_{0}\|_{H^{2m}}+
  \|w^{2}F[u,w]\|_{X^{1,2;m}_{T}}.
\end{equation}
We must estimate
\begin{equation*}
  \textstyle
  \|w^{2}F[u,w]\|_{X^{1,2;m}_{T}}=
  \sum_{2j+|\alpha|\le 2m}
  \|\partial^{j}_{t}\partial^{\alpha}_{x}(w^{2}F[u,w])\|
  _{L^{1}_{T}L^{2}}.
\end{equation*}
The derivative can be expanded as a finite sum
\begin{equation*}
  \textstyle
  \partial^{j}_{t}\partial^{\alpha}_{x}(w^{2}F[u,w])=
  \sum\int_{0}^{1}
  G \cdot W_{1}W_{2} U_{1}\cdot \dots \cdot U_{\nu}d \sigma
\end{equation*}
where $0\le \nu\le 2m$ and
\begin{itemize}
  \item 
  $W_{1}=\partial^{h_{1}}_{t}\partial^{\alpha_{1}}_{x}w$
  and
  $W_{2}=\partial^{h_{2}}_{t}\partial^{\alpha_{2}}_{x}w$
  \item $U_{k}=
    \partial^{j_{k}}_{t}\partial^{\beta_{k}}_{x}(u+\sigma w)$
  \item $h_{1}+h_{2}+j_{1}+\dots +j_{\nu}=j$,
  $\alpha_{1}+\alpha_{2}+\beta_{1}+\dots +\beta_{\nu}=\alpha$
  \item it is not restrictive to assume that
  $2j_{\nu}+|\beta_{\nu}|\ge 2j_{k}+|\beta_{k}|$ for
  all $k$ and that $2h_{2}+|\alpha_{2}|\ge 2h_{1}+|\alpha_{1}|$
  \item 
  $G$ satisfies $|G|\lesssim|u+\sigma w|^{p-\nu-2}$
  so that 
  $\|G\|_{L^{\infty}}\lesssim
    (\|u\|_{L^{\infty}}+\|w\|_{L^{\infty}})^{p-\nu-2}$.
\end{itemize}
We take the $L^{2}$ norm in $x$ of each product.
We consider two cases.

\textsc{First case:}
$2j_{\nu}+|\beta_{\nu}|\ge 2h_{2}+|\alpha_{2}|$.
Then we write
\begin{equation*}
\begin{split}
  \|GW_{1}W_{2}U_{1}\dots U_{\nu}\|_{L^{2}}
  \le
  &
  \|G\|_{L^{\infty}}
  \|W_{1}\|_{L^{\infty}}
  \|W_{2}\|_{L^{\infty}}
  \|U_{1}\|_{L^{\infty}}
  \dots
  \|U_{\nu-1}\|_{L^{\infty}}
  \|U_{\nu}\|_{L^{2}}
  \\
  \lesssim
  \|G\|_{L^{\infty}}
  &
  \|W_{1}\|_{W^{N_{r},r}}
  \|W_{2}\|_{W^{N_{r},r}}
  \|U_{1}\|_{H^{N_{2}}}
  \dots
  \|U_{\nu-1}\|_{H^{N_{2}}}
  \|U_{\nu}\|_{L^{2}}
\end{split}
\end{equation*}
by Sobolev embedding, where
\begin{equation*}
  \textstyle
  N_{r}=\lfloor \frac nr \rfloor+1,
  \qquad
  N_{2}=\lfloor \frac n2 \rfloor+1;
\end{equation*}
note that since the maximal order of derivation is 
$2j_{\nu}+|\beta_{\nu}|\le 2m$, we have
$2j_{k}+|\beta_{k}|\le m$ for $k<\nu$
and $2h_{i}+|\alpha_{i}|\le m$, $i=1,2$,
so that, using \eqref{eq:boundsHN}, we have
\begin{equation*}
  \|U_{i}\|_{L^{\infty}_{T}H^{N_{2}}}\le
  \|u\|_{X^{\infty,2;(m+N_{2})/2}}+
  \|w\|_{X^{\infty,2;(m+N_{2})/2}}\le
  \|u\|_{X^{\infty,2;m}}+\|w\|_{X^{\infty,2;m}}
\end{equation*}
\begin{equation*}
  \le
  C_{m}+M_{m}(T)
\end{equation*}
provided $m\ge N_{2}$, where 
$C_{m}=C(\|u_{0}\|_{H^{2m}},\|xu_{0}\|_{L^{2}})$.
Moreover we can estimate
\begin{equation*}
  \|G\|_{L^{\infty}}\lesssim
  (\|u\|_{L^{\infty}}+\|w\|_{L^{\infty}})^{p-\nu-2}
  \lesssim
  (\bra{t}^{-1}+\|w\|_{L^{\infty}})^{p-\nu-2}.
\end{equation*}
This gives, for $t\in[0,T]$,
\begin{equation*}
  \|GW_{1}W_{2}U_{1}\dots U_{\nu}\|_{L^{2}}\le
  (C_{m}+M_{m}(T))^{\nu}
  (\bra{t}^{-1}+\|w\|_{L^{\infty}})^{p-\nu-2}
  \|W_{1}\|_{W^{N_{r},r}}
  \|W_{2}\|_{W^{N_{r},r}}.
\end{equation*}
We now take the $L^{1}$ norm in $t\in[0,T]$.
Since $2h_{i}+|\alpha_{i}|\le m$, $i=1,2$,
we can write
\begin{equation*}
  \textstyle
  \int_{0}^{T}
  \bra{t}^{2+\nu-p}
  \|W_{1}\|_{W^{N_{r},r}}
  \|W_{2}\|_{W^{N_{r},r}}dt
  \le
  \|\bra{t}^{2+\nu-p}\|_{L^{\frac{q}{q-2}}}
  \|w\|_{X_{T}^{q,r;(m+N_{r})/2}}^{2}
\end{equation*}
and if $m\ge N_{r}$ (which is implied by $m\ge N_{2}$)
this gives
\begin{equation*}
  \le \|\bra{t}^{2+\nu-p}\|_{L^{\frac{q}{q-2}}}
  M_{m}(T)^{2}.
\end{equation*}
We choose $q\in(2,\infty)$ such that $(q,r)$ is admissible,
i.e. $q=\frac{4r}{n(r-2)}$. Since $\nu\le 2m$,
we see that
\begin{equation*}
  \textstyle
  \|\bra{t}^{2+\nu-p}\|_{L^{\frac{q}{q-2}}}\le
  \|\bra{t}^{2+2m-p}\|_{L^{\frac{q}{q-2}}}<\infty
  \quad\text{provided}\quad 
  p>2m+3-n \frac{r-2}{2r}
\end{equation*}
and in this case
\begin{equation}\label{eq:Istr}
  \textstyle
  \int_{0}^{T}
  \bra{t}^{2+\nu-p}
  \|W_{1}\|_{W^{N_{r},r}}
  \|W_{2}\|_{W^{N_{r},r}}dt
  \lesssim M_{m}(T)^{2}.
\end{equation}
With a similar computation we can write
\begin{equation*}
  \textstyle
  \int_{0}^{T}
  \|w\|_{L^{\infty}}^{p-\nu-2}
  \|W_{1}\|_{W^{N_{r},r}}
  \|W_{2}\|_{W^{N_{r},r}}dt
  \lesssim
  \left\|\|w\|_{L^{\infty}}^{p-\nu-2}\right\|
  _{L_{T}^{\frac{q}{q-2}}}
  M_{m}(T)^{2}
\end{equation*}
and
\begin{equation*}
  \left\|\|w\|_{L^{\infty}}^{p-\nu-2}\right\|
  _{L_{T}^{\frac{q}{q-2}}}
  \le
  \|w\|_{L^{\infty}_{T}L^{\infty}}^{p-q-\nu}
  \|w\|_{L^{q}_{T}L^{\infty}}^{q-2}
  \lesssim
  \|w\|_{L^{\infty}_{T}H^{N_{2}}}^{p-q-\nu}
  \|w\|_{L^{q}_{T}W^{N_{r},r}}^{q-2}
  \lesssim
  M_{m}(T)^{p-q-\nu}
\end{equation*}
provided $p>2m+q$. In conclusion we have
\begin{equation}\label{eq:IIstr}
  \textstyle
  \int_{0}^{T}
  \|w\|_{L^{\infty}}^{p-\nu-2}
  \|W_{1}\|_{W^{N_{r},r}}
  \|W_{2}\|_{W^{N_{r},r}}dt
  \lesssim
  M_{m}(T)^{p-q-\nu}.
\end{equation}
Combining \eqref{eq:Istr}, \eqref{eq:IIstr} we conclude
\begin{equation}\label{eq:IIIstr}
  \|GW_{1}W_{2}U_{1}\dots U_{\nu}\|_{L^{2}}
  \lesssim
  M_{m}(T)^{2}+M_{m}(T)^{p}
\end{equation}
provided $p>2m+q+2$ (so that $p-q-\nu>2$).

\textsc{Second case:} 
$2j_{\nu}+|\beta_{\nu}|< 2h_{2}+|\alpha_{2}|$.
Then we estimate the product as follows
\begin{equation*}
\begin{split}
  \|GW_{1}W_{2}U_{1}\dots U_{\nu}\|_{L^{2}}
  \le
  &
  \|G\|_{L^{\infty}}
  \|W_{1}\|_{L^{\infty}}
  \|W_{2}\|_{L^{2}}
  \|U_{1}\|_{L^{\infty}}
  \dots
  \|U_{\nu-1}\|_{L^{\infty}}
  \|U_{\nu}\|_{L^{\infty}}
  \\
  \lesssim
  &
  \|G\|_{L^{\infty}}
  \|W_{1}\|_{W^{N_{r},r}}
  \|W_{2}\|_{L^{2}}
  \|U_{1}\|_{H^{N_{2}}}
  \dots
  \|U_{\nu-1}\|_{H^{N_{2}}}
  \|U_{\nu}\|_{H^{N_{2}}}
\end{split}
\end{equation*}
and in this case the maximal order of derivation is
$2h_{2}+|\alpha_{2}|$ so that we have 
$2j_{k}+|\beta_{k}|\le m$ for all $k$
and $2h_{1}+|\alpha_{1}|\le m$.
Proceeding in a similar way as in the first case, we get
again \eqref{eq:IIIstr}.

Summing up, and recalling \eqref{eq:basic}, we have
proved
\begin{equation}\label{eq:lstr}
  M_{m}(T)\lesssim\|w_{0}\|_{H^{2m}}+
  M_{m}(T)^{2}+M_{m}(T)^{p}
\end{equation}
with an implicit constant depending on
$\|u_{0}\|_{H^{2m}}+\|xu_{0}\|_{L^{2}}$, provided
\begin{equation}\label{eq:condp}
  \textstyle
  m=\lfloor  \frac n2 \rfloor+1,
  \qquad
  p>2m+q+2
\end{equation}
and $(q,r)$ is admissible with $q\in(2,\infty)$.
We can take $q>2$ arbitrarily close to 2, so that
it is sufficient to assume
\begin{equation*}
  p>n+6
\end{equation*}
to achieve \eqref{eq:condp}.
Finally, a standard continuation argument shows that
if $\|w_{0}\|_{H^{2m}}$ is sufficiently small
with respect to $\|u_{0}\|_{H^{2m}}+\|xu_{0}\|_{L^{2}}$,
the (maximal) local solution $w$ to the equation
\eqref{eq:diffeq} can be continued for all times, and
this proves the claim.

\section{Weak--strong uniqueness}\label{sec:weak_stro_uniq}

We recall the definition of the energy $E(u(t))$
\begin{equation*}
  E(u(t))=E(u)=\frac12\int_{\Omega}|\nabla_{x}u|^{2}dx+
  \frac{1}{p+1}\int_{\Omega}|u|^{p+1}dx
\end{equation*}
of a solution $u(t,x)$ to the Cauchy problem
\begin{equation}\label{eq:pertpb2}
  iu_{t}+\Delta u=|u|^{p-1}u,
  \qquad
  u(0,x)=u_{0},
  \qquad
  u(t,\cdot)\vert_{\partial \Omega}=0.
\end{equation}
Following \cite{Struwe06}, we prove a general stability result
for local solutions of \eqref{eq:pertpb2}, from which the
uniqueness Theorem \ref{the:uniqueness} follows
immediately.

\begin{theorem}[]\label{the:stabil}
  Let $I$ be an open interval containing $[0,T]$, $T>0$.
  Let $u,v$ be two distributional solutions to \eqref{eq:pertpb2} 
  on $I \times \Omega$ such that
  \begin{equation*}
    u\in 
    C(I;H^{2}(\Omega))\cap
    C^{1}(I;H^{1}_{0}(\Omega))\cap
    C^{2}(I;L^{2}(\Omega))\cap
    L^{\infty}(I \times \Omega),
    \quad
    \Delta u\in C(I;H^{1}_{0}(\Omega)),
  \end{equation*}
  \begin{equation*}
    v\in 
    C(I;H^{1}_{0}(\Omega))\cap
    C^{1}(I;L^{2}(\Omega)).
  \end{equation*}
  Assume in addition that $v$ satisfies an energy inequality
  \begin{equation*}
    E(v(t))\le E(v(0)).
  \end{equation*}
  Then the difference $w=v-u$ satisfies the energy estimate
  \begin{equation*}
    E(w(t))\le C e^{Ct}(E(w(0))+\|w(0)\|_{L^{2}(\Omega)}^{2}),
    \qquad
    t\in[0,T]
  \end{equation*}
  where $C$ is a constant depending on
  \begin{equation}\label{eq:const}
    C=C(p,T,\|u\|_{L^{\infty}([0,T]\times \Omega)}).
  \end{equation}
\end{theorem}

To prove Theorem \ref{the:stabil}, consider the
difference $w=v-u$, which satisfies the equation
\begin{equation*}
  iw_{t}+\Delta w=|u+w|^{p-1}(u+w)-|u|^{p-1}u.
\end{equation*}
We prepare an estimate for the $L^{2}$ norm of $w$.
Using the multiplier $i \overline{w}$ we get
\begin{equation*}
\begin{split}
  \textstyle
  \partial_{t}\|w(t)\|_{L^{2}}^{2}=
  &
  \textstyle
  2\Im\int_{\omega}(|u+w|^{p-1}(u+w)-|u|^{p-1}u)\overline{w}dx
  \\
  \le & \textstyle
  C(\|u\|_{L^{\infty}_{I}L^{\infty}})\cdot
  \int_{\Omega}(|w|^{2}+|w|^{p+1})dx
  \\
  \le & \textstyle
  C[\|w\|_{L^{2}}^{2}+E(w(t))]
\end{split}
\end{equation*}
and by Gronwall's Lemma
\begin{equation}\label{eq:L2estb}
  \textstyle
  \|w(t)\|_{L^{2}}^{2}\le
  C\|w(0)\|_{L^{2}}^{2}
  +C\int_{0}^{t}e^{C(t-s)}E(w(s))ds,
  \qquad t\in [0,T]
\end{equation}
with $C=C(T,\|u\|_{L^{\infty}_{T}L^{\infty}})$.

Next, we split
\begin{equation*}
  E(v)=E(u)+A(t)+B(t)
\end{equation*}
where
\begin{equation*}
  \textstyle
  A(t)=
  \frac 12
  \int |\nabla_{x}w|^{2}dx
  +
  \int (\frac{|u+w|^{p+1}-|u|^{p+1}}{p+1}-|u|^{p-1}
  \Re(u \overline{w}))dx
\end{equation*}
\begin{equation*}
  \textstyle
  B(t)=
  \Re
  \int(\nabla u \cdot\nabla \overline{w}+|u|^{p-1}
  u \overline{w})dx.
\end{equation*}
Since $E(u)=E(u(0))$ and $E(v)\le E(v(0))$ we have
\begin{equation}\label{eq:finalg}
  0\le E(v(0))-E(v(t))=A(0)-A(t)+B(0)-B(t).
\end{equation}
Writing
\begin{equation*}
  \phi(t)=\frac{|t|^{\frac{p+1}{2}}}{p+1}
  \quad\text{so that}\quad 
  \frac{|u|^{p+1}}{p+1}=\phi(|u|^{2}),
\end{equation*}
we see that
\begin{equation*}
  \textstyle
  \partial_{\sigma}\phi(|u+\sigma w|^{2})=
  \phi'(|u+\sigma w|^{2})2\Re((u+\sigma w)\overline{w})=
  \Re(|u+\sigma w|^{p-1}(u+\sigma w)\overline{w}),
\end{equation*}
\begin{equation*}
\begin{split}
  \partial^{2}_{\sigma}\phi(|u+\sigma w|^{2})
  =&
  (p-1)|u+\sigma w|^{p-3}\Re((u+\sigma w)\overline{w})^{2}+
  |u+\sigma w|^{p-1}|w|^{2}
  \\
  \ge&
  |u+\sigma w|^{p-1}|w|^{2}
  \ge
  2^{2-p}\sigma^{p-1}|w|^{p+1}-|u|^{p-1}|w|^{2}
\end{split}
\end{equation*}
(since $p\ge3$). We get easily
\begin{equation*}
  \textstyle
  \frac{|u+w|^{p+1}-|u|^{p+1}}{p+1}-|u|^{p-1}\Re(u \overline{w})
  =
  \int_{0}^{1}\int_{0}^{\sigma}
  \partial^{2}_{\tau}\phi(|u+\tau w|^{2})
  d \tau d \sigma
\end{equation*}
\begin{equation*}
  \ge \textstyle
  \frac{2^{2-p}}{p(p+1)}|w|^{p+1}-\frac 12|u|^{p-1}|w|^{2}
\end{equation*}
which implies
\begin{equation*}
  \textstyle
  A(t)\ge \frac{1}{p2^{p}} E(w(t))-C\|w\|_{L^{2}}^{2},
  \qquad
  C=\frac 12\|u\|_{L^{\infty}_{T}L^{\infty}}^{p-1}.
\end{equation*}
Recalling \eqref{eq:L2estb}, this gives
for $t\in [0,T]$
\begin{equation*}
  \textstyle
  A(t)\ge \frac{1}{p2^{p}}E(w(t))-
  C\int_{0}^{t}E(w(s))ds
  -C\|w(0)\|_{L^{2}}^{2}
\end{equation*}
for some $C=C(T,\|u\|_{L^{\infty}_{T}L^{\infty}})$.
On the other hand
\begin{equation*}
  \textstyle
  \frac{|u+w|^{p+1}-|u|^{p+1}}{p+1}-|u|^{p-1}\Re(u \overline{w})
  \le
  C(\|u\|_{L^{\infty}_{T}L^{\infty}})(|w|^{p+1}+|w|^{2})
\end{equation*}
which implies
\begin{equation*}
  A(0)\le CE(w(0))+C\|w(0)\|_{L^{2}}^{2}
\end{equation*}
and in conclusion
\begin{equation}\label{eq:partial1}
  \textstyle
  A(0)-A(t)\le
  -\frac{1}{p2^{p}}E(w(t))+
      C\int_{0}^{t}E(w(s))ds
      +C\|w(0)\|_{L^{2}}^{2}
\end{equation}
with $C=C(T,\|u\|_{L^{\infty}_{T}L^{\infty}})$.

In order to estimate $B(t)$, we first remark the following.
If $W(t,x)$, $U(t,x)$ satisfy
\begin{equation*}
  iW_{t}+\Delta W=F,
  \qquad
  iU_{t}+\Delta U=G
\end{equation*}
with Dirichled boundary conditions,
then for any $\chi(t)\in C^{\infty}_{c}((0,T))$ we have formally
\begin{equation}\label{eq:distrib}
  \textstyle
  \iint_{\Omega} \chi'(t)(\nabla U \cdot \nabla \overline{W})dxdt=
  \iint_{\Omega} \chi(t)
  (U_{t}\overline{F}+\overline{W}_{t}G)dxdt
\end{equation}
Identity \eqref{eq:distrib} is obvious for smooth 
$U,W$, by integration by parts.
By approximation, \eqref{eq:distrib} holds also if
$W$ is a solution of $iW_{t}+\Delta W=F$ in 
$\mathscr{D}'((0,T)\times \Omega)$, with
\begin{equation*}
  W\in  C^{1}([0,T];H^{-1}(\Omega))\cap C([0,T];H^{1}_{0}(\Omega)),
\end{equation*}
so that $F\in C([0,T];H^{-1}(\Omega))$, and
\begin{equation*}
  U\in C^{1}([0,T];H^{1}_{0}(\Omega))\cap C([0,T];H^{2}(\Omega))
  \quad\text{with}\quad 
  \Delta U\in C([0,T];H^{1}_{0}(\Omega))
\end{equation*}
so that $G\in C([0,T];H^{1}_{0}(\Omega))$).
Consider now a sequence of test functions 
$\chi_{k}(t)\in C^{\infty}_{c}((0,T))$, non negative,
such that $\chi_{k}\uparrow\one{[0,t]}$ pointwise, 
$t\in(0,T]$, and write
\begin{equation*}
  \textstyle
  B(t)-B(0)=\lim_{k\to \infty} I_{k},
  \qquad
  I_{k}:=\Re\iint_{\Omega} \chi_{k}'(t)
  (\nabla u \cdot \overline{\nabla w}
  +|u|^{p-1}u\overline{w})dx dt.
\end{equation*}
Using \eqref{eq:distrib} with the choices $W=w$, $U=u$,
$F=|u+w|^{p-1}(u+w)-|u|^{p-1}u$ and
$G=|u|^{p-1}u$, we get
\begin{equation*}
  \textstyle
  I_{k}=\Re
  \iint_{\Omega}
  \chi_{k}(t)
  \left[
    u_{t}(|u+w|^{p-1}(\overline{u}+\overline{w})-
    |u|^{p-1}\overline{u})
    + \overline{w}_{t}|u|^{p-1}u-
    \partial_{t}(|u|^{p-1}u \overline{w})
  \right]
\end{equation*}
\begin{equation*}
  \textstyle
  =\Re\iint_{\Omega}\chi_{k}(t)
  \left[
    u_{t}(|u+w|^{p-1}(\overline{u}+\overline{w})-
    |u|^{p-1}\overline{u})-
    \overline{w}\partial_{t}(|u|^{p-1}u)
  \right]dxdt.
\end{equation*}
We compute
\begin{equation*}
  \textstyle
  \partial_{t}(|u|^{p-1}u)=
  |u|^{p-3}(|u|^{2}u_{t}+
  \frac{p-1}{2}u \overline{u}_{t}+
  \frac{p-1}{2}\overline{u} u_{t}
  )
\end{equation*}
so that
\begin{equation*}
  \textstyle
  \Re [\overline{w}\partial_{t}(|u|^{p-1}u)]=
  \Re(H u_{t}),
  \qquad
  H=
  \frac 12
  |u|^{p-3}(
  (p+1)|u|^{2}\overline{w}+
  (p-1)\overline{u}^{2}w)
\end{equation*}
and
\begin{equation*}
  \textstyle
  I_{k}=\Re\iint_{\Omega}\chi_{k}(t)
  \left[
    |u+w|^{p-1}(\overline{u}+\overline{w})-
    |u|^{p-1}\overline{u}-H
  \right]u_{t}
  dxdt.
\end{equation*}
We have
\begin{equation*}
\begin{split}
  |u+w&|^{p-1}(\overline{u}+\overline{w})-
  |u|^{p-1}\overline{u}-H
  \\
  =&
  (|u+w|^{p-1}-|u|^{p-1})(\overline{u}+\overline{w})-
  \textstyle
  \frac{p-1}{2}|u|^{p-3}(|u|^{2}w+\overline{u}^{2}w)
  \\
  =&
  \textstyle
  \frac{p-1}{2}
  \int_{0}^{1}
  (|u+\sigma w|^{p-3}-|u|^{p-3})(|u|^{2}\overline{w}
    +\overline{u}^{2}w)d \sigma+R
\end{split}
\end{equation*}
where
\begin{equation*}
  \textstyle
  R=
  \frac{p-1}{2}
  \int_{0}^{1}
  |u+\sigma w|^{p-3}
  [2 \sigma|w|^{2}(\overline{u}+\overline{w})+u \overline{w}^{2}
    + \overline{u}|w|^{2}]d \sigma.
\end{equation*}
We have easily (if $p\ge 4$)
\begin{equation*}
  |R|\le C(\|u\|_{L^{\infty}_{T}L^{\infty}}^{p-1})
  (|w|^{2}+|w|^{p}),
  \qquad
  ||u+\sigma w|^{p-3}-|u|^{p-3}|\le
  C(\|u\|_{L^{\infty}_{T}L^{\infty}}^{p-1})
  (|w|+|w|^{p-3})
\end{equation*}
and summing up
\begin{equation*}
  \textstyle
  |I_{k}|\le 
  C(\|u\|_{L^{\infty}_{T}L^{\infty}}^{p-1})
  \iint_{\Omega}\chi_{k}(|w|^{2}+|w|^{p+1}).
\end{equation*}
Letting $k\to \infty$ we deduce
\begin{equation*}
  \textstyle
  |B(t)-B(0)|\le 
  C\int_{0}^{t}[E(w(s))+\|w(s)\|_{L^{2}(\Omega)}^{2}]ds,
  \qquad
  C=C(\|u\|_{L^{\infty}_{T}L^{\infty}})
\end{equation*}
and using \eqref{eq:L2estb} we have
\begin{equation*}
  \textstyle
  B(0)-B(t)\le C\int_{0}^{t}E(w(s))ds
  +C\|w(0)\|_{L^{2}}^{2}.
\end{equation*}
Recalling \eqref{eq:finalg} and \eqref{eq:partial1}
we obtain
\begin{equation*}
  \textstyle
  E(w(t))\le
  C\int_{0}^{t}E(w(s))ds+CE(w(0))+C\|w(0)\|_{L^{2}}^{2}
\end{equation*}
with $C=C(\|u\|_{L^{\infty}_{T}L^{\infty}})$ as usual,
and by Gronwall's Lemma we conclude the proof.

% b_f_post
%%% >>> BIBLATEX:
% \printbibliography
%%% >>> BIBTEX: (cancellare \usepackage{biblatex})
% \bibliography{/Users/piero/Documents/Biblioteca/-bib/bibliodatabase.bib}
\bibliographystyle{abbrv}
%%% >>> entrambe si possono usare contemporaneamente a:

\end{document}